\documentclass[11pt,bezier]{article}
\usepackage{amsmath,amssymb,enumerate,amsfonts,setspace,graphicx}

\usepackage{tkz-berge}

\newtheorem{prethm}{{\bf  Theorem}}

\newenvironment{thm}{\begin{prethm}{\hspace{-0.5
               em}{\bf .}}}{\end{prethm}}

\newtheorem{prepro}{{\bf  Theorem}}

\newenvironment{pro}{\begin{prepro}{\hspace{-0.5
               em}{\bf .}}}{\end{prepro}}
\newtheorem{precor}{{\bf  Corollary}}

\newenvironment{cor}{\begin{precor}{\hspace{-0.5
               em}{\bf .}}}{\end{precor}}
\newtheorem{preconj}{{\bf  Conjecture}}

\newtheorem{preremark}{{\bf  Remark}}

\newenvironment{remark}{\begin{preremark}{\hspace{-0.5
               em}{\bf .}}}{\end{preremark}}
\newtheorem{prelem}{{\bf  Lemma}}

\newenvironment{lem}{\begin{prelem}{\hspace{-0.5
               em}{\bf .}}}{\end{prelem}}

\newtheorem{preexample}{{\bf  Example}}

\newenvironment{exa}{\begin{preexample}{\hspace{-0.5
               em}{\bf .}}}{\end{preexample}}

\newtheorem{preproof}{{\bf  Proof.}}

\newenvironment{proof}[1]{\begin{preproof}{\rm
               #1}\hfill{$\Box$}}{\end{preproof}}

\title{\large \bf On the intersection graph of ideals of a commutative ring
\thanks
{
{\indent ~ {2010{ \it Mathematics Subject Classification}: 05C15, 05C17, 05C69, 13A99, 13C99.}}\newline
}
}

\author{\bf\small\sc F. Heydari\\
{\footnotesize {\em Department of
 Mathematics, Karaj Branch,
Islamic Azad University, Karaj, Iran}}\\
{\footnotesize {\em  f-heydari@kiau.ac.ir }}
}

\date{}
\begin{document}
\maketitle

\begin{abstract}
Let $R$ be a commutative ring and $M$ be an $R$-module, and let $I(R)^*$ be the set of all non-trivial ideals of $R$.
The $M$-intersection graph of ideals of $R$, denoted by $G_M(R)$,
is a graph with the vertex set $I(R)^*$, and two distinct vertices $I$ and $J$ are adjacent if and only if
$IM\cap JM\neq 0$. For every multiplication $R$-module $M$, the diameter and the girth of $G_M(R)$ are determined. Among other results, we prove that if $M$ is a faithful $R$-module and
the clique number of $G_M(R)$ is finite, then $R$ is a semilocal ring. We denote the $\mathbb{Z}_n$-intersection graph of ideals of the ring $\mathbb{Z}_m$ by $G_n(\mathbb{Z}_m)$, where $n,m\geq 2$ are integers and $\mathbb{Z}_n$ is a $\mathbb{Z}_m$-module. We determine the values of $n$ and $m$ for which $G_n(\mathbb{Z}_m)$ is perfect. Furthermore, we derive a sufficient condition for $G_n(\mathbb{Z}_m)$ to be weakly perfect.
\end{abstract}

\section{Introduction}
Let $R$ be a commutative ring, and $I(R)^*$ be the set of all non-trivial ideals of $R$.
There are many papers on assigning a graph to a ring $R$, for instance see [1--4].
Also the intersection graphs of some algebraic structures such as groups, rings and modules have been studied by several authors, see \cite{akbtaval, india, Cs}.
 In \cite{india}, the intersection graph of ideals of $R$, denoted by $G(R)$, was introduced as the graph with vertices $I(R)^*$ and for distinct $I,J\in I(R)^*$,
 the vertices $I$ and $J$ are adjacent if and only if $I\cap J\neq 0$. Also in \cite{akbtaval}, the intersection
graph of submodules of an $R$-module $M$, denoted by $G(M)$, is defined to be the graph whose vertices are
the non-trivial submodules of $M$ and two distinct vertices are
adjacent if and only if they have non-zero intersection.
 In this paper, we generalize $G(R)$ to $G_M(R)$, the \textit{$M$-intersection graph of ideals} of $R$, where $M$ is an $R$-module.\\
\indent Throughout the paper, all rings are commutative with non-zero identity and all modules are unitary.
A module is called a {\it uniform} module if the intersection of any two non-zero submodules is non-zero. An $R$-module $M$ is said to be a \textit{multiplication} module if every submodule of $M$ is of the form $IM$, for some ideal $I$ of $R$.
The \textit{annihilator} of $M$ is denoted by $ann(M)$. The module $M$ is called a \textit{faithful} $R$-module if $ann(M)=0$. By a non-trivial submodule of
$M$, we mean a non-zero proper submodule of $M$. Also, $J(R)$ denotes
the Jacobson radical of $R$ and $Nil(R)$ denotes the ideal of all nilpotent elements of $R$. By $\mathrm{Max}(R)$, we denote the set of all maximal ideals of $R$.
A ring having only finitely many maximal ideals is said to be a \textit{semilocal} ring.
As usual, $\mathbb{Z}$ and $\mathbb{Z}_n$ will denote the integers and the integers modulo $n$, respectively.\\
\indent A graph in which any two distinct vertices are adjacent is called a \textit{complete graph}. We denote the complete graph on $n$ vertices by $K_n$. A \textit{null graph} is a graph containing no edges. Let $G$ be a graph. The {\it complement} of $G$ is denoted by $\overline{G}$. The set of vertices and the set of edges of $G$ are
denoted by $V(G)$ and $E(G)$, respectively. A subgraph $H$ of $G$ is said to be an \textit{induced subgraph} of $G$ if it has exactly the edges that appear
in $G$ over $V(H)$. Also, a subgraph $H$
of $G$ is called a \textit{spanning subgraph} if $V(H)=V(G)$. Suppose that $x,y\in V(G)$. We denote by $deg(x)$ the degree of a vertex $x$ in $G$. A {\it regular graph} is a graph where each vertex has the same degree. We recall
that a \textit{walk} between $x$ and $y$ is a sequence $x=v_0$ --- $v_1$ --- $\cdots$ --- $v_k=y$ of vertices of $G$ such that for every $i$ with $1\leq i \leq k$,
the vertices $v_{i-1}$ and $v_i$ are adjacent. A \textit{path} between $x$ and $y$ is a walk between $x$ and $y$ without repeated vertices. We say that $G$ is
\textit{connected} if there is a path between any two distinct
 vertices of $G$.
For vertices $x$ and $y$ of $G$, let $d(x,y)$ be the length of a shortest path from $x$ to $y$ ($d(x,x)=0$ and $d(x,y)=\infty$ if there is no path between $x$
and $y$). The \textit{diameter} of $G$, $diam(G)$, is the supremum of the set $\{d(x,y) : x \ \hbox{and} \ y \ \hbox{are vertices of} \ G\}$.
The \textit{girth} of $G$, denoted by $gr(G)$, is the length of a shortest cycle in $G$ ($gr(G)=\infty$ if $G$ contains no cycles).
A \textit{clique} in $G$ is a set of pairwise adjacent vertices and the number of vertices in the largest clique of $G$,
denoted by $\omega(G)$, is called the \textit{clique number} of $G$. The \textit{chromatic number} of $G$, $\chi(G)$, is the minimal number of colors which can be assigned to the vertices of $G$ in such a way that every two adjacent vertices have different colors. A graph $G$ is {\it perfect} if for every induced subgraph $H$ of $G$, $\chi(H)=\omega(H)$. Also, $G$ is called {\it weakly perfect} if $\chi(G)=\omega(G)$.\\
\indent In the next section, we introduce the $M$-intersection graph of ideals of $R$, denoted by $G_M(R)$, where $R$ is a commutative ring and $M$ is a
non-zero $R$-module. It is shown that for every multiplication $R$-module $M$, $diam(G_M(R))\in \{0,1,2,\infty\}$ and $gr(G_M(R))\in \{3,\infty\}$. Among other results, we prove that if $M$ is a faithful $R$-module and
$\omega(G_M(R))$ is finite, then $|\mathrm{Max}(R)|\leq \omega(G_M(R))+1$ and $J(R)=Nil(R)$. In the last section, we consider the $\mathbb{Z}_n$-intersection graph of ideals of $\mathbb{Z}_m$, denoted by $G_n(\mathbb{Z}_m)$, where $n,m\geq 2$ are integers and $\mathbb{Z}_n$ is a $\mathbb{Z}_m$-module. We show that $G_n(\mathbb{Z}_m)$ is a perfect graph if and only if $n$ has at most four distinct prime divisors. Furthermore, we derive a sufficient condition for $G_n(\mathbb{Z}_m)$ to be weakly perfect. As a corollary, it is shown that the intersection graph of ideals of $\mathbb{Z}_m$ is weakly perfect, for every integer $m\geq 2$.

\section{The $M$-intersection graph of ideals of $R$}

In this section, we introduce the $M$-intersection graph of ideals of $R$ and study its basic properties.\\

\noindent{\bf Definition.}
{\rm
Let $R$ be a commutative ring and $M$ be a non-zero $R$-module. The \textit{$M$-intersection graph of ideals} of $R$, denoted by $G_M(R)$, is the
graph with vertices $I(R)^*$ and two distinct vertices $I$ and $J$ are adjacent if and only if $IM\cap JM\neq 0$.\\
}

Clearly, if $R$ is regarded as a module over itself, that is, $M=R$, then the $M$-intersection graph of ideals of $R$ is exactly the same as the intersection
graph of ideals of $R$. Also, if $M$ and $N$ are two isomorphic $R$-modules, then $G_M(R)$ is the same as $G_N(R)$.

\begin{exa}
{\rm
Let $R=\mathbb{Z}_{12}$. Then we have the following graphs.
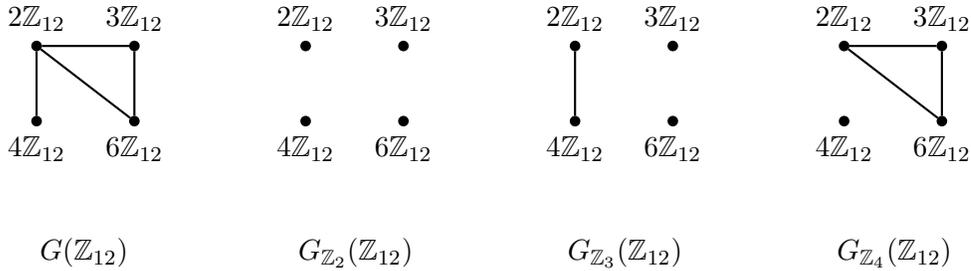
\begin{center}
    \begin{tikzpicture}
        \GraphInit[vstyle=Classic]
        \Vertex[x=1,y=0,style={black,minimum size=3pt},LabelOut=true,Lpos=270,L=$4\mathbb{Z}_{12}$]{4}
        \Vertex[x=1,y=1,style={black,minimum size=3pt},LabelOut=true,Lpos=90,L=$2\mathbb{Z}_{12}$]{2}
        \Vertex[x=2.3,y=0,style={black,minimum size=3pt},LabelOut=true,Lpos=270,L=$6\mathbb{Z}_{12}$]{6}
        \Vertex[x=2.3,y=1,style={black,minimum size=3pt},LabelOut=true,Lpos=90,L=$3\mathbb{Z}_{12}$]{3}
        \Edges(2,3)
        \Edges(6,3)
        \Edges(2,6)
        \Edges(2,4)

    \end{tikzpicture}
\hspace{1cm}
    \begin{tikzpicture}
        \GraphInit[vstyle=Classic]
        \Vertex[x=1,y=0,style={black,minimum size=3pt},LabelOut=true,Lpos=270,L=$4\mathbb{Z}_{12}$]{4}
        \Vertex[x=1,y=1,style={black,minimum size=3pt},LabelOut=true,Lpos=90,L=$2\mathbb{Z}_{12}$]{2}
        \Vertex[x=2.3,y=0,style={black,minimum size=3pt},LabelOut=true,Lpos=270,L=$6\mathbb{Z}_{12}$]{6}
        \Vertex[x=2.3,y=1,style={black,minimum size=3pt},LabelOut=true,Lpos=90,L=$3\mathbb{Z}_{12}$]{3}

    \end{tikzpicture}
\hspace{1cm}
    \begin{tikzpicture}
        \GraphInit[vstyle=Classic]
        \Vertex[x=1,y=0,style={black,minimum size=3pt},LabelOut=true,Lpos=270,L=$4\mathbb{Z}_{12}$]{4}
        \Vertex[x=1,y=1,style={black,minimum size=3pt},LabelOut=true,Lpos=90,L=$2\mathbb{Z}_{12}$]{2}
        \Vertex[x=2.3,y=0,style={black,minimum size=3pt},LabelOut=true,Lpos=270,L=$6\mathbb{Z}_{12}$]{6}
        \Vertex[x=2.3,y=1,style={black,minimum size=3pt},LabelOut=true,Lpos=90,L=$3\mathbb{Z}_{12}$]{3}
        \Edges(2,4)

    \end{tikzpicture}
\hspace{1cm}
    \begin{tikzpicture}
        \GraphInit[vstyle=Classic]
        \Vertex[x=1,y=0,style={black,minimum size=3pt},LabelOut=true,Lpos=270,L=$4\mathbb{Z}_{12}$]{4}
        \Vertex[x=1,y=1,style={black,minimum size=3pt},LabelOut=true,Lpos=90,L=$2\mathbb{Z}_{12}$]{2}
        \Vertex[x=2.3,y=0,style={black,minimum size=3pt},LabelOut=true,Lpos=270,L=$6\mathbb{Z}_{12}$]{6}
        \Vertex[x=2.3,y=1,style={black,minimum size=3pt},LabelOut=true,Lpos=90,L=$3\mathbb{Z}_{12}$]{3}
        \Edges(2,3)
        \Edges(6,3)
        \Edges(2,6)

    \end{tikzpicture}
 \end{center}
\hspace{1.4cm}
$G(\mathbb{Z}_{12})$ \hspace{2cm}
$G_{\mathbb{Z}_2}(\mathbb{Z}_{12})$ \hspace{1.8cm}
$G_{\mathbb{Z}_3}(\mathbb{Z}_{12})$ \hspace{1.8cm}
$G_{\mathbb{Z}_4}(\mathbb{Z}_{12})$
}
\end{exa}

\begin{exa}
{\rm
 Let $n\geq 2$ be an integer. If $[m_1,m_2]$ is the least common multiple of two distinct integers $m_1,m_2\geq 2$, then
$m_1\mathbb{Z}\mathbb{Z}_n\cap m_2\mathbb{Z}\mathbb{Z}_n=m_1\mathbb{Z}_n\cap m_2\mathbb{Z}_n=[m_1,m_2]\mathbb{Z}_n$.
Thus $m_1\mathbb{Z}$ and $m_2\mathbb{Z}$ are adjacent in $G_{\mathbb{Z}_n}(\mathbb{Z})$ if and only if $n$ does not divide $[m_1,m_2]$.
 }
\end{exa}

\begin{exa}
{\rm
Let $p$ be a prime number and $n,m$ be two positive integers. If $p^n$ divides $m$, then $m\mathbb{Z}$ is an isolated vertex of $G_{\mathbb{Z}_{p^n}}(\mathbb{Z})$.
Therefore, since $\mathbb{Z}_{p^n}$ is a uniform $\mathbb{Z}$-module, so
$G_{\mathbb{Z}_{p^n}}(\mathbb{Z})$ is a disjoint union of an infinite complete graph and its complement.
Also, $\mathbb{Z}_{p^\infty}$ (the {\it quasi-cyclic} $p$-group),
is a uniform $\mathbb{Z}$-module and $ann(\mathbb{Z}_{p^\infty})=0$. Hence $G_{\mathbb{Z}_{p^\infty}}(\mathbb{Z})$ is an infinite complete graph.
}
\end{exa}

\begin{remark}
{\rm
Obviously, if $M$ is a faithful multiplication $R$-module, then $G_M(R)$ is a complete graph if and only if $M$ is
a uniform $R$-module.
}
\end{remark}

\begin{remark}
\label{rem1}
{\rm
Let $R$ be a commutative ring and let $M$ be a non-zero $R$-module.
\begin{enumerate}[\rm(1)]

\item If $M$ is a faithful $R$-module, then $G(R)$ is a spanning subgraph of $G_M(R)$. To see this, suppose that $I$ and $J$ are adjacent vertices
of $G(R)$. Then $I\cap J\neq 0$ implies that $(I\cap J)M\neq 0$ and so $IM\cap JM\neq 0$. Therefore $I$ is adjacent to $J$ in $G_M(R)$.

\item If $M$ is a multiplication $R$-module, then $G(M)$ is an induced subgraph of $G_M(R)$. Note that for each non-trivial submodule $N$ of $M$,
there is a non-trivial ideal $I$ of $R$, such that $N=IM$ and so we can assign $N$ to $I$. Also, $N=IM$ is adjacent to $K=JM$ in $G(M)$ if and only if
$IM\cap JM\neq 0$, that is, if and only if $I$ is adjacent to $J$ in $G_M(R)$.
\end{enumerate}
}
\end{remark}

\begin{thm}
{Let $R$ be a commutative ring and let $M$ be a faithful $R$-module. If $G_M(R)$ is not connected, then $M$ is a direct sum of two $R$-modules.}
\end{thm}

\begin{proof}
{Suppose that $C_1$ and $C_2$ are two distinct components of
$G_M(R)$. Let $I\in C_1$ and $J\in C_2$. Since $M$ is a faithful $R$-module, so $IM\cap JM=0$ implies that
$I\nsubseteq J$ and $J\nsubseteq I$. Now if $I+J\neq R$, then $I$ --- $I+J$ --- $J$ is a path between $I$ and $J$, a contradiction. Thus $I+J= R$ and so
$M=IM\oplus JM$.}
\end{proof}

The next theorem shows that for every multiplication $R$-module $M$, the diameter of $G_M(R)$ has $4$ possibilities.

\begin{thm}
\label{multicon}
{Let $R$ be a commutative ring and $M$ be a multiplication $R$-module. Then $diam(G_M(R))\in \{0,1,2,\infty\}$.}
\end{thm}

\begin{proof}
{Assume that $G_M(R)$ is a connected graph with at least two vertices. So $M$ is a faithful module. If there is a non-trivial ideal $I$ of $R$ such that $IM=M$, then $I$ is adjacent to all other vertices. Hence $diam(G_M(R))\leq 2$. Otherwise, we claim that $G(M)$ is connected.
Let $N$ and $K$ be two distinct vertices of $G(M)$. Since $M$ is a multiplication module, so $N=IM$ and $K=JM$, for some non-trivial ideals
$I$ and $J$ of $R$. Suppose that $I=I_1$ --- $I_2$ --- $\cdots$ --- $I_n=J$ is a path between $I$ and $J$ in $G_M(R)$. Therefore,
$N$ --- $I_2M$ --- $\cdots$ --- $I_{n-1}M$ --- $K$ is a walk between $N$ and $K$. Thus, we conclude that there is also a path between $N$ and $K$ in $G(M)$.
The claim is proved. So by \cite[Theorem 2.4]{akbtaval}, $diam(G(M))\leq 2$. Now, suppose that $I_1$ and $I_2$ are two distinct vertices of $G_M(R)$. If $I_1M\cap I_2M=0$, then $I_1M$ and $I_2M$ are two distinct vertices of $G(M)$. Hence there exists a non-trivial submodule $N$ of $M$ which is adjacent to both $I_1M$ and $I_2M$ in $G(M)$. Since $M$ is a multiplication module, so $N=JM$, for some non-trivial ideal $J$ of $R$. Thus $J$ is adjacent to both $I_1$ and $I_2$ in $G_M(R)$. Therefore $diam(G_M(R))\leq 2$.
}
\end{proof}

\begin{thm}
Let $R$ be a commutative ring and $M$ be a multiplication $R$-module. If $G_M(R)$ is a connected regular graph of finite degree, then $G_M(R)$ is a complete graph.
\end{thm}

\begin{proof}
{Suppose that $G_M(R)$ is a connected regular graph of finite degree. If $ann(M)\neq 0$, then $G_M(R)=K_1$. So assume that $ann(M)=0$. We claim that $M$ is an Artinian module. Suppose to the contrary that $M$ is not an Artinian module. Then there is a
descending chain $I_1M\supset I_2M\supset \cdots \supset I_nM\supset \cdots$ of submodules of $M$, where $I_i$'s are non-trivial ideals of $R$. This implies that $deg(I_1)$ is infinite, a contradiction. The claim is proved. Therefore $M$ has at least one minimal submodule. To complete the proof, it suffices to show that $M$ contains a unique minimal submodule. By contrary, suppose that $N_1$ and $N_2$ are two distinct minimal submodules of $M$. Hence $N_1=I_1M$ and $N_2=I_2M$, where $I_1$ and $I_2$ are two non-trivial ideals of $R$. Since $N_1\cap N_2=0$, so $I_1$ and $I_2$ are not adjacent. By Theorem \ref{multicon}, there is a vertex $J$ which is adjacent to both $I_1$ and $I_2$. So both
$I_1M$ and $I_2M$ are
contained in $JM$. Thus each vertex adjacent to $I_1$ is adjacent to $J$ too. This implies that
$deg(J) > deg(I_1)$, a contradiction.
}
\end{proof}

Also, the following theorem shows that for every multiplication $R$-module $M$, the girth of $G_M(R)$ has $2$ possibilities.

\begin{thm}
{ Let $R$ be a commutative ring and $M$ be a multiplication $R$-module. Then $gr(G_M(R))\in\{3,\infty\}$.}
\end{thm}

\begin{proof}
{
Suppose that $I_1$ --- $I_2$ --- $\cdots$ --- $I_n$ --- $I_1$ is a cycle of length $n$ in $G_M(R)$. If $n=3$, we are done.
Thus assume that $n\geq4$. Since $I_1M\cap I_2M\neq 0$ and $M$ is a multiplication module, we have $I_1M\cap I_2M=JM$, where $J$ is
a non-zero ideal of $R$. If $J$ is a proper ideal of $R$ and $J\neq I_1,I_2$, then $I_1$ --- $J$ --- $I_2$ --- $I_1$ is a triangle in $G_M(R)$.
Otherwise, we conclude that $I_1M\subseteq I_2M$ or $I_2M\subseteq I_1M$. Similarly, we can assume that $I_{i}M\subseteq I_{i+1}M$ or $I_{i+1}M\subseteq I_{i}M$, for every $i$,
$1<i<n$. Without loss of generality suppose that $I_1M\subseteq I_2M$. Now, if $I_2M\subseteq I_3M$, then $I_1$ --- $I_2$ --- $I_3$ --- $I_1$ is a cycle
of length 3 in $G_M(R)$. Therefore assume that $I_3M\subseteq I_2M$. Since $I_3M\subseteq I_4M$ or $I_4M\subseteq I_3M$, so $I_2$ --- $I_3$ --- $I_4$ --- $I_2$ is a triangle in $G_M(R)$. Hence if $G_M(R)$ contains a cycle, then $gr(G_M(R))=3$.
}
\end{proof}

\begin{lem}
\label{lemiso}
{ Let $R$ be a commutative ring and $M$ be a non-zero $R$-module. If $I$ is an isolated vertex of $G_M(R)$, then the following hold:
\begin{enumerate}[\rm(1)]
\item $I$ is a maximal ideal of $R$ or $I\subseteq ann(M)$.

\item If $I\nsubseteq ann(M)$, then $I=Ra$, for every $a\in I\setminus ann(M)$.
\end{enumerate}
}
\end{lem}

\begin{proof}
{ $(1)$ There is a maximal ideal $\mathfrak{m}$ of $R$ such that $I\subseteq\mathfrak{m}$. Assume that $I\neq \mathfrak{m}$. Then we have
$IM=IM\cap \mathfrak{m}M=0$, since $I$ is an isolated vertex. So $I\subseteq ann(M)$.\\
\indent $(2)$ Suppose that $a\in I\setminus ann(M)$ and $I\neq Ra$. Since $I$ is an isolated vertex, we have $RaM=IM\cap RaM=0$ and so $a\in ann(M)$, a contradiction.
Thus $I=Ra$.}
\end{proof}

\begin{thm}
Let $R$ be a commutative ring and $M$ be a faithful $R$-module. If $G_M(R)$ is a null graph, then it has at most two vertices and $R$ is isomorphic to one of the
following rings:
\begin{enumerate}[\rm(1)]
\item $F_1\times F_2$, where $F_1$ and $F_2$ are fields;

\item $F[x]/(x^2)$, where $F$ is a field;

\item $L$, where $L$ is a coefficient ring of characteristic $p^2$, for some prime number $p$.
\end{enumerate}
\end{thm}

\begin{proof}
{By Lemma \ref{lemiso}, every non-trivial ideal of $R$ is maximal and so by \cite[Theorem 1.1]{maxring1}, $R$ cannot have more than two different
non-trivial ideals. Thus $G_M(R)$ has at most two vertices. Also, by \cite[Theorem 4]{maxring2}, $R$ is isomorphic to one of the
mentioned rings.
}
\end{proof}

In the next theorem we show that if $M$ is a faithful $R$-module and
$\omega(G_M(R))<\infty$, then $R$ is a semilocal ring.

\begin{thm}
{ Let $R$ be a commutative ring and $M$ be a faithful $R$-module. If $\omega(G_M(R))$ is finite then
$|\mathrm{Max}(R)|\leq \omega(G_M(R))+1$ and $J(R)=Nil(R)$.}
\end{thm}

\begin{proof}{
First we prove that $|\mathrm{Max}(R)|\leq \omega(G_M(R))+1$. Let $\omega=\omega(G_M(R))$. By contradiction, assume that
$\mathfrak{m}_1,\ldots,\mathfrak{m}_{\omega+2}$ are distinct maximal ideals of $R$. We know that $\mathfrak{m}_1\cdots\mathfrak{m}_i\neq0$, for every $i$,
$1\leq i\leq \omega+1$. Otherwise, $\mathfrak{m}_1\cdots\mathfrak{m}_j=0$, for some $j$, $1\leq j\leq \omega+1$. So
$\mathfrak{m}_1\cdots\mathfrak{m}_j\subseteq \mathfrak{m}_{j+1}$ and hence by Prime Avoidance Theorem \cite[Proposition 1.11]{att},
we have $\mathfrak{m}_t\subseteq \mathfrak{m}_{j+1}$, for some $t$, $1\leq t\leq j$, which is impossible. This implies that
$\{\mathfrak{m}_1,\mathfrak{m}_1\mathfrak{m}_2,\ldots,\mathfrak{m}_1\cdots\mathfrak{m}_{\omega+1}\}$ is a clique in $G_M(R)$, a contradiction.
Thus $|\mathrm{Max}(R)|\leq \omega+1$.\\
\indent Now, we prove that $J(R)=Nil(R)$. By contrary, suppose that $a\in J(R)\setminus Nil(R)$. Since $Ra^{i}M\cap Ra^{j}M\neq 0$, for every $i,j$, $i<j$ and $\omega(G_M(R))$ is finite,
we conclude that $Ra^{t}= Ra^{s}$, for some integers $t<s$. Hence $a^{t}(1-ra^{s-t})=0$, for some $r\in R$. Since $a\in J(R)$, so $1-ra^{s-t}$ is a unit.
This yields that $a^{t}=0$, a contradiction. The proof is complete.}
\end{proof}

\section {The $\mathbb{Z}_n$-intersection graph of ideals of $\mathbb{Z}_m$}

Let $n,m\geq 2$ be two integers and $\mathbb{Z}_n$ be a $\mathbb{Z}_m$-module. In this section we study the $\mathbb{Z}_n$-intersection graph of ideals of the ring $\mathbb{Z}_m$. Also, we generalize some results given in \cite{nik}. For abbreviation, we denote $G_{\mathbb{Z}_n}(\mathbb{Z}_m)$ by $G_n(\mathbb{Z}_m)$.
Clearly, $\mathbb{Z}_n$ is a $\mathbb{Z}_m$-module if and only if $n$ divides $m$. \\
\indent Throughout this section, without loss of generality, we assume that $m=p_1^{\alpha_1}\cdots p_s^{\alpha_s}$ and $n=p_1^{\beta_1}\cdots p_s^{\beta_s}$, where $p_i$'s are distinct primes, $\alpha_i$'s are positive integers, $\beta_i$'s are non-negative integers, and $0\leq \beta_i\leq \alpha_i$ for $i=1,\ldots,s$. Let $S=\{1,\ldots,s\}$ and $S'=\{i\in S : \beta _i\neq 0\}$. The cardinality of $S'$ is denoted by $s'$.
For two integers $a$ and $b$, we write $a|b$ ($a\nmid b$) if $a$ divides $b$ ($a$ does not divide $b$).\\
\indent First we have the following remarks.

\begin{remark}
\label{zngzm}
{\rm It is easy to see that $I(\mathbb{Z}_m)=\{d\mathbb{Z}_m : d$ divides $m \}$ and
$|I(\mathbb{Z}_m)^{*}|=\prod_{i=1}^{s}(\alpha_i+1)-2$. Let $\mathbb{Z}_n$ be a $\mathbb{Z}_m$-module. If $n|d$, then $d\mathbb{Z}_m$ is an isolated vertex of $G_n(\mathbb{Z}_m)$.
Obviously, $d_1\mathbb{Z}_m$ and $d_2\mathbb{Z}_m$ are adjacent if and only if $n\nmid [d_1,d_2]$. This implies that $G_n(\mathbb{Z}_m)$ is a subgraph of $G(\mathbb{Z}_m)$.
}
\end{remark}

\begin{remark}
\label{clique}
{\rm  Let $\mathbb{Z}_n$ be a $\mathbb{Z}_m$-module and $d=p_1^{r_1}\cdots p_s^{r_s}(\neq 1,m)$ be a divisor of $m$. We set $D_d=\{i\in S :  r_i< \beta_i \}$. Clearly, $D_d\subseteq S'$. Suppose that $W$ is a clique of $G_n(\mathbb{Z}_m)$. Then $\Gamma_W=\{D_d : d\mathbb{Z}_m\in W \}$ is an intersecting family of subsets of $S'$. (A family of sets is intersecting if any two of its sets have a non-empty intersection.) Also, if $\Gamma$ is an intersecting family of subsets of $S'$ and $W_\Gamma=\{ d\mathbb{Z}_m : d\neq 1,m, \ d|m, \ D_d\in \Gamma \}$ is non-empty, then $W_\Gamma$ is a clique of
$G_n(\mathbb{Z}_m)$. (If $D$ is a non-empty subset of $S'$ and $\Gamma=\{D\}$, then we will denote $W_\Gamma$ by $W_D$.) Thus we have
\begin{center}
$\omega(G_n(\mathbb{Z}_m))=max \left\{ |W_\Gamma|  : \  \Gamma \ \hbox{is an intersecting family of subsets of} \ S' \right\}.$
\end{center}
}
\end{remark}

Now, we provide a lower bound for the clique number of $G_n(\mathbb{Z}_m)$.

\begin{thm}
\label{bound}
{ Let $\mathbb{Z}_n$ be a $\mathbb{Z}_m$-module. Then
\begin{center}
$\omega(G_n(\mathbb{Z}_m))\geq max \left\{\beta_j\prod_{i\neq j}(\alpha_i+1)-1 \   :  \  \beta_j\neq 0 \right\}.$
\end{center}
}
\end{thm}

\begin{proof}
{ Suppose that $\beta_j\neq 0$. With the notations of the previous remark, let $\Gamma=\{ D\subseteq S' : j\in D \}$. Then
$\Gamma$ is an intersecting family of subsets of $S'$ and so $W_\Gamma$
is a clique of $G_n(\mathbb{Z}_m)$. Clearly, $|W_\Gamma|=\beta_j\prod_{i\neq j}(\alpha_i+1)-1$.
Therefore $\omega(G_n(\mathbb{Z}_m))\geq \beta_j\prod_{i\neq j}(\alpha_i+1)-1$ and hence the result holds.
}
\end{proof}

Clearly, if $n=p_1^{\beta_1}$ $(\beta_1>1)$, then equality holds in the previous theorem. Also, if $n$ has only two distinct prime divisors, that is, $s'=2$, then again equality holds. So the lower bound is sharp.

\begin{exa}
{\rm
 Let $m=n=p_1^2p_2^2p_3^2$, where $p_1, p_2, p_3$ are distinct primes. Thus $S'=S=\{1,2,3\}$ and $G_n(\mathbb{Z}_m)=G(\mathbb{Z}_m)$. It is easy to see that $|W_{\{1\}}|=|W_{\{2\}}|=|W_{\{3\}}|=2$ and $|W_{\{1,2\}}|=|W_{\{1,3\}}|=|W_{\{2,3\}}|=4$. Also, $|W_{\{1,2,3\}}|=7$. Let $\Gamma_j=\{ D\subseteq S' : j\in D \}$, for $j=1,2,3$. Hence $|W_{\Gamma_j}|=17$, for $j=1,2,3$. If $\Gamma=\left\{\{1,2\}, \{1,3\}, \{2,3\}, \{1,2,3\}\right\}$, then $|W_\Gamma|=19$. Therefore $\omega(G(\mathbb{Z}_m))=19$.
}
\end{exa}

By the strong perfect graph theorem, we determine the values of $n$ and $m$ for which $G_n(\mathbb{Z}_m)$ is a perfect graph.

\begin{pro}
\label{Berge}
{\rm (The Strong Perfect Graph Theorem \cite{strong}) A finite graph $G$ is perfect if and only if neither $G$ nor $\overline{G}$ contains an induced odd cycle of length at least $5$.}
\end{pro}

\begin{thm}
Let $\mathbb{Z}_n$ be a $\mathbb{Z}_m$-module. Then $G_n(\mathbb{Z}_m)$ is perfect if and only if $n$ has at most four distinct prime divisors.
\end{thm}

\begin{proof}
{First suppose that $s'\geq5$ and $n=p_1^{\beta_1}\cdots p_{s'}^{\beta_{s'}}$, where $p_i$'s are distinct primes and $\beta_i$'s are positive integers.
Let $D_1=\{p_1,p_5\}$, $D_2=\{p_1,p_2\}$, $D_3=\{p_2,p_3\}$, $D_4=\{p_3,p_4\}$, and $D_5=\{p_4,p_5\}$. Now, assume that $d_i\mathbb{Z}_m\in W_{D_i}$, for $i=1,\ldots,5$. Hence $d_1\mathbb{Z}_m$ --- $d_2\mathbb{Z}_m$ --- $d_3\mathbb{Z}_m$ --- $d_4\mathbb{Z}_m$ --- $d_5\mathbb{Z}_m$ --- $d_1\mathbb{Z}_m$ is an induced cycle of length 5 in $G_n(\mathbb{Z}_m)$. So by Theorem \ref{Berge}, $G_n(\mathbb{Z}_m)$ is not a perfect graph.\\
\indent Conversely, suppose that $G_n(\mathbb{Z}_m)$ is not a perfect graph. Then by Theorem \ref{Berge}, we have the following cases:\\
\indent Case 1. $d_1\mathbb{Z}_m$ --- $d_2\mathbb{Z}_m$ --- $d_3\mathbb{Z}_m$ --- $d_4\mathbb{Z}_m$ --- $d_5\mathbb{Z}_m$ --- $d_1\mathbb{Z}_m$ is an induced cycle of length 5 in $G_n(\mathbb{Z}_m)$. Let $D_i=D_{d_i}$, for $i=1,\ldots,5$. So $D_5\cap D_1\neq \varnothing$ and $D_i\cap D_{i+1}\neq \varnothing$, for $i=1,\ldots,4$. Let $p_{5}\in D_5\cap D_1$ and $p_{i}\in D_i\cap D_{i+1}$, for $i=1,\ldots,4$. Clearly, $p_1,\ldots,p_5$ are distinct and thus $s'\geq 5$.\\
\indent Case 2. $d_1\mathbb{Z}_m$ --- $d_2\mathbb{Z}_m$ --- $d_3\mathbb{Z}_m$ --- $d_4\mathbb{Z}_m$ --- $d_5\mathbb{Z}_m$ --- $d_6\mathbb{Z}_m$ is an induced path of length 5 in $G_n(\mathbb{Z}_m)$. Let $D_i=D_{d_i}$, for $i=1,\ldots,6$. So $D_i\cap D_{i+1}\neq \varnothing$, for $i=1,\ldots,5$. Let $p_i\in D_i\cap D_{i+1}$, for $i=1,\ldots,5$. Clearly, $p_1,\ldots,p_5$ are distinct and hence $s'\geq 5$.\\
\indent Case 3. There is an induced cycle of length 5 in $\overline{G_n(\mathbb{Z}_m)}$. So $G_n(\mathbb{Z}_m)$ contains an induced cycle of length 5 and by Case 1, we are done.\\
\indent Case 4. $d_1\mathbb{Z}_m$ --- $d_2\mathbb{Z}_m$ --- $d_3\mathbb{Z}_m$ --- $d_4\mathbb{Z}_m$ --- $d_5\mathbb{Z}_m$ --- $d_6\mathbb{Z}_m$ is an induced path of length 5 in $\overline{G_n(\mathbb{Z}_m)}$. Since $D_{d_1}\cap D_{d_3}\neq \varnothing$, $D_{d_1}\cap D_{d_4}\neq \varnothing$ and $D_{d_3}\cap D_{d_4}=\varnothing$, we may assume that $\{p_1,p_2\}\subseteq D_{d_1}$, where $p_1\in D_{d_3}$ and $p_2\in D_{d_4}$, for some distinct $p_1,p_2\in S'$. Similarly, we find that $\{p_3,p_4\}\subseteq D_{d_2}$, for some distinct $p_3,p_4\in S'\setminus \{p_1,p_2\}$ and also $|D_{d_3}|\geq 2$. Now, since $D_{d_3}\cap D_{d_2}=\varnothing$ and $p_2\notin D_{d_3}$, we deduce that $s'\geq 5$.
}
\end{proof}

\begin{cor}
{The graph $G(\mathbb{Z}_m)$ is perfect if and only if $m$ has at most four distinct prime divisors.
}
\end{cor}

In the next theorem, we derive a sufficient condition for $G_n(\mathbb{Z}_m)$ to be weakly perfect.

\begin{thm}
\label{perfect}
{Let $\mathbb{Z}_n$ be a $\mathbb{Z}_m$-module. If $\alpha_i\leq 2\beta_i-1$ for each $i\in S'$, then $G_n(\mathbb{Z}_m)$ is weakly perfect.
}
\end{thm}

\begin{proof}
{Let $D$ be a non-empty subset of $S'$ and $\overline{D}=S'\setminus D$. As we mentioned in Remark \ref{clique}, if $W_D$ is non-empty, then $W_D$ is a clique of $G_n(\mathbb{Z}_m)$. Also, the vertices of $W_{S'}$ (if $W_{S'}\neq \varnothing$) are adjacent to all non-isolated vertices. Suppose that $D_1$ and $D_2$ are two non-empty subsets of $S'$ and $D_1\subseteq D_2$. Since $\alpha_i\leq 2\beta_i-1$ for each $i\in S'$, so $\prod_{i\in D_2\setminus D_1}(\alpha_i-\beta_i+1)\leq \prod_{i\in D_2\setminus D_1}\beta_i$. This implies that $\prod_{i\in D_1}\beta_i\prod_{i\notin D_1}(\alpha_i-\beta_i+1)\leq \prod_{i\in D_2}\beta_i\prod_{i\notin D_2}(\alpha_i-\beta_i+1)$ and hence $|W_{D_1}|\leq |W_{D_2}|$. \\
\indent Let $\Gamma$ be an intersecting family of subsets of $S'$ and $\omega(G_n(\mathbb{Z}_m))=|W_\Gamma|$.
Let $D\subseteq S'$. We show that $D\in \Gamma$ or $\overline{D}\in \Gamma$. Assume that $D\notin \Gamma$. So there is $D_1\in \Gamma$ such that $D\cap D_1=\varnothing$. Thus $D_1\subseteq \overline{D}$ and hence $\overline{D}\in \Gamma$.
We claim that $|W_{\overline{D}}|\leq |W_D|$, for each $D\in \Gamma$. Suppose to the contrary, $D\in \Gamma$ and $|W_{\overline{D}}|>|W_D|$. If $A\in \Gamma$ and $A\subseteq D$, then $\overline{D}\subseteq \overline{A}$. So we have $|W_A|\leq|W_D|<|W_{\overline{D}}|\leq|W_{\overline{A}}|$. Let $\Phi=\Gamma \cup\{\overline{A} : A\in \Gamma, A\subseteq D\}\setminus \{A\in \Gamma : A\subseteq D\}$. Then $\Phi$ is an intersecting family of subsets of $S'$ and $|W_\Gamma|<|W_\Phi|$, a contradiction. The claim is proved.\\
\indent Now, we show that $G_n(\mathbb{Z}_m)$ has a proper $|W_\Gamma|$-vertex coloring. First we color all vertices of $W_\Gamma$ with different colors.
Next we color each family $W_D$ of vertices out of $W_\Gamma$ with colors of vertices of $W_{\overline{D}}$. Note that if $D\notin \Gamma$, then $\overline{D}\in \Gamma$ and $|W_D|\leq |W_{\overline{D}}|$. Suppose that $d_1\mathbb{Z}_m$ and $d_2\mathbb{Z}_m$ are two adjacent vertices of $G_n(\mathbb{Z}_m)$. Thus $D_{d_1}\cap D_{d_2}\neq \varnothing$.
Without loss of generality, one can assume
$D_{d_1}\neq D_{d_2}$. So we deduce that $\overline{D_{d_1}}\neq \overline{D_{d_2}}$ and $D_{d_1}\neq \overline{D_{d_2}}$.
Therefore, $d_1\mathbb{Z}_m$ and $d_2\mathbb{Z}_m$ have different colors. Thus $\chi(G_n(\mathbb{Z}_m))\leq |W_\Gamma|$ and hence $\omega(G_n(\mathbb{Z}_m))=\chi(G_n(\mathbb{Z}_m))=|W_\Gamma|$.
}
\end{proof}

As an immediate consequence of the previous theorem, we have the
next result.

\begin{cor}
{The graph $G(\mathbb{Z}_m)$ is weakly perfect, for every integer $m\geq 2$.
}
\end{cor}

In the case that $\alpha_i=2\beta_i-1$ for each $i\in S'$, we determine the exact value of $\chi(G_n(\mathbb{Z}_m))$. It is exactly the lower bound obtained in the Theorem \ref{bound}.

\begin{thm}
{Let $\mathbb{Z}_n$ be a $\mathbb{Z}_m$-module. If $\alpha_i=2\beta_i-1$ for each $i\in S'$, then $\omega(G_n(\mathbb{Z}_m))=\chi(G_n(\mathbb{Z}_m))=2^{s'-1}\prod_{i\in S'}\beta_i\prod_{i\in S\setminus S'}(\alpha_i+1)-1$.
}
\end{thm}

\begin{proof}
{Let $D\neq \varnothing$ be a proper subset of $S'$. Then $|W_D|=\prod_{i\in D}\beta_i\prod_{i\notin D}(\alpha_i-\beta_i+1)=\prod_{i\in S'}\beta_i\prod_{i\in S\setminus S'}(\alpha_i+1)$ and hence $|W_D|=|W_{\overline{D}}|$. Also, the vertices of $W_{S'}$ (if $W_{S'}\neq \varnothing$) are adjacent to all non-isolated vertices and $|W_{S'}|=\prod_{i\in S'}\beta_i\prod_{i\in S\setminus S'}(\alpha_i+1)-1$. Clearly if $\Gamma$ is an intersecting family of subsets of $S'$, then $|\Gamma|\leq 2^{s'-1}$. Moreover, if $\beta_j\neq 0$ and $\Gamma_j=\{ D\subseteq S' : j\in D \}$, then $|\Gamma_j|=2^{s'-1}$. Thus by Theorem \ref{perfect}, $\omega(G_n(\mathbb{Z}_m))=\chi(G_n(\mathbb{Z}_m))=|W_{\Gamma_j}|=2^{s'-1}\prod_{i\in S'}\beta_i\prod_{i\in S\setminus S'}(\alpha_i+1)-1$.
}
\end{proof}

\begin{cor}
{Let $m=p_1\cdots p_s$, where $p_i$'s are distinct primes. Then $\omega(G(\mathbb{Z}_m))=\chi(G(\mathbb{Z}_m))=2^{s-1}-1$.
}
\end{cor}

We close this article by the following problem.

\noindent {\bf Problem.}
Let $\mathbb{Z}_n$ be a $\mathbb{Z}_m$-module. Then is it true that $G_n(\mathbb{Z}_m)$ is a weakly perfect graph?

\end{document}